\documentclass[11pt]{article}

\usepackage{amsmath}
\usepackage{amsfonts,amssymb,amstext}
\usepackage{amsmath}
\usepackage{amstext}
\usepackage{amsthm}
\usepackage{graphicx}
\usepackage{booktabs}
\usepackage{xcolor}
\makeatletter

\newcommand{\beq}{\begin{equation}}
\newcommand{\eeq}{\end{equation}}
\newcommand{\bea}{\begin{eqnarray}}
\newcommand{\eea}{\end{eqnarray}}
\newcommand{\beaa}{\begin{eqnarray*}}
\newcommand{\eeaa}{\end{eqnarray*}}

\newcommand{\n}{\noindent}
\newcommand{\q}{\quad}
\newcommand{\qq}{\qquad}

\newcommand{\E}{\rm I\!E}

\newcommand{\I}{\varphi}

\newcommand{\de}{\delta}
\newcommand{\De}{\Delta}

\newcommand{\al}{\alpha}
\newcommand{\la}{\lambda}
\newcommand{\f}{\infty}
\newcommand{\vs}{\varepsilon}
\newcommand{\cd}{\cdot}
\newcommand{\si}{\sigma}

\newcommand{\be}{\beta}

\newcommand{\Om}{\Omega}

\newcommand{\inl}{\int_{-\pi}^{\pi}}

\newcommand {\s} {\section}
\newcommand {\sn} {\subsection}

\newtheorem{thm}{Theorem}[section]
\newtheorem{lem}{Lemma}[section]
\newtheorem{pp}{Proposition}[section]
\newtheorem{cor}{Corollary}[section]

\newtheorem{exa}{Example}[section]
\newtheorem{rem}{Remark}[section]

\newtheorem{den}{Definition}[section]

\numberwithin{equation}{section}

\newtheorem*{TA}{Theorem A}
\newtheorem*{TB}{Theorem B}
\newtheorem*{TC}{Theorem C}

\usepackage[active]{srcltx}

\begin{document}


\title{On asymptotic behavior of the prediction error for
a class of deterministic stationary sequences}

\author{N. M. Babayan\thanks{ Russian-Armenian University, Yerevan, Armenia,
e-mail: nmbabayan@gmail.com} \,
and M. S. Ginovyan\thanks{Boston University, Boston, USA, e-mail: ginovyan@math.bu.edu. Corresponding author.}}

\date{\today}

\maketitle

\begin{abstract}
\noindent
One of the main problem in prediction theory of 
stationary
processes $X(t)$ is to describe the asymptotic behavior of the best linear mean
squared prediction error in predicting $X(0)$ given $ X(t),$ $-n\le t\le-1$,
as $n$ goes to infinity.
This behavior depends on the regularity (deterministic or non-deterministic)
of the process $X(t)$.
In his seminal paper {\it 'Some purely deterministic processes' (J. of Math. and
Mech.,} {\bf 6}(6), 801-810, 1957), for a specific spectral density 
that has a very high order contact with zero M. Rosenblatt showed that
the prediction error behaves like a power as $n\to\f$.
In the paper Babayan et al. {\it 'Extensions of Rosenblatt's results on
the asymptotic behavior of the prediction error for deterministic stationary
sequences' (J. Time Ser. Anal.} {\bf 42}, 622-652, 2021), Rosenblatt's result was
extended to the class of spectral densities of the form $f=f_dg$,
where $f_d$ is the spectral density of a deterministic process that
has a very high order contact with zero, while $g$ is a function
that can have polynomial type singularities.
In this paper, we describe new extensions of the above quoted results
in the case where the function $g$ can have {\it arbitrary power type singularities}.
Examples illustrate the obtained results.
\end{abstract}



\s{Introduction}
\label{Int}

Let $X(t),$ $t\in\mathbb{Z}: = \{0,\pm1,\ldots\}$, be a centered discrete-time
second-order stationary process. The process is assumed to have an absolutely
continuous spectrum with spectral density function $f(\la),$
$\la\in [-\pi, \pi].$ The {\it 'finite' linear prediction problem}
is as follows.

Suppose we observe a finite realization of the process $X(t)$:
$\{X(t), \,\, -n\le t\le-1\},$ $n\in\mathbb{N}: = \{1,2, \ldots\}.$
We want to make an one-step ahead prediction, that is, to predict the
unobserved random variable $X(0)$,  using the {\it linear predictor}
$Y=\sum_{k=1}^{n}c_kX(-k)$.

The coefficients
$c_k$, $k=1,2,\ldots,n$, are chosen so as to minimize
{\it the mean-squared error}: $\E\left|X(0) - Y\right|^2,$
where $\E[\cd]$ stands for the expectation operator.
If such minimizing constants $\hat c_k:=\hat c_{k,n}$ can be found,
then the random variable $\hat X_n(0):=\sum_{k=1}^{n}\hat c_kX(-k)$ is called
{\it the best linear one-step ahead predictor} of $X(0)$
based on the observed finite past: $X(-n), \ldots, X(-1)$.
The minimum mean-squared error:
$$\si_{n}^2(f): =\E\left|X(0) - \hat X_n(0)\right|^2\geq0
$$
is called {\it the best linear one-step ahead prediction error} of $X(t)$
based on the past of length $n$.

One of the main problem in prediction theory of second-order stationary
processes, called {\it the 'direct' prediction problem} is to describe the
asymptotic behavior of the prediction error $\sigma_n^2(f)$ as $n\to\f$.
This behavior depends on the regularity nature (deterministic or nondeterministic)
of the observed process $X(t)$.

Observe that $\sigma_{n+1}^2(f)\leq \sigma_n^2(f)$  ($n\in\mathbb{N}$),
and hence, the limit of $\sigma_n^2(f)$ as $n\to\f$ exists.
Denote by $\sigma^2(f): = \sigma_\infty^2(f)$ the prediction error
of $X(0)$ by the entire infinite past: $\{X(t)$, $t\le-1\}$.

From the prediction point of view it is natural to distinguish the class of
processes for which we have {\it error-free prediction} by the entire infinite past, that is,
$\si^2(f)=0$. Such processes are called {\it deterministic} or {\it singular}.
Processes for which $\si^2(f)>0$ are called {\it nondeterministic}
(for more about these terms see, e.g., Babayan et al. \cite{BGT},
and Grenander and Szeg\H{o} \cite{GS}, p.176).


Define the {\it relative prediction error} $\de_n(f): = \si^2_n(f) - \si^2(f),$
and observe that
$\delta_n(f)$ is non-negative and tends to zero as $n\to\infty$.
But what about the speed of convergence of $\delta_n(f)$ to zero as $n\to\infty$?
The paper deals with this question.
Specifically, the prediction problem we are interested in is
{\it to describe the rate of decrease of $\delta_n(f)$ to zero as}
$n \to \infty,$ depending on the regularity
nature (deterministic or nondeterministic) of the observed process $X(t)$.

The prediction problem stated above goes back to classical works of
A. Kolmogorov, G. Szeg\H{o} and N. Wiener. It was then considered by many authors
for different classes of nondeterministic processes
(see, e.g., the survey papers Bingham \cite{Bin1} and Ginovyan \cite{G-4},
and references therein).

We focus in this paper on deterministic processes, that is, when  $\si^2(f) =0$.
This case is not only of theoretical interest, but is also important from the point
of view of applications.
For example, as pointed out by Rosenblatt \cite{Ros} (see also Pierson \cite{Pl}),
situations of this type arise in Neumann's theoretical model of
storm-generated ocean waves. Such models are also of interest
in meteorology (see Fortus \cite{For}).

Only few works are devoted to the study of the speed of convergence of
$\de_n(f)=\si^2_n(f)$ to zero as $n\to\infty$, that is, the asymptotic behavior of
the prediction error for deterministic processes.
One needs to go back to the classical work of Rosenblatt \cite{Ros},
where the asymptotic behavior of the prediction error $\si^2_n(f)$
was investigated in the following two cases:

(a) the spectral density $f(\la)$ is continuous and positive on an
interval of the segment $[-\pi,\pi]$ and zero elsewhere,

(b) the spectral density $f(\la)$ has a very high order of contact with zero
at points $\la=0, \pm\pi$, and is strictly positive otherwise.

For the case (a) above, Rosenblatt \cite{Ros} proved
that the prediction error $\si^2_n(f)$ decreases to zero exponentially
as $n\to\f$.
Later the problem (a) was studied by  Babayan  \cite{Bb-1,Bb-2} and
Babayan et al. \cite{BGT} (see also Davisson \cite{Dav1} and Fortus \cite{For}),
where some extensions of Rosenblatt's result have been obtained.

Concerning the case (b) above, for a specific deterministic process $X(t)$, 
Rosenblatt proved in \cite{Ros} that
the prediction error $\si^2_n(f)$ decreases to zero {\it like a power} as $n\to\f$. 
More precisely, the deterministic process $X(t)$ considered in
Rosenblatt \cite{Ros} has the spectral density
\beq
\label{nd4}
f_a(\la):=\frac{e^{(2\la-\pi)\varphi(\la)}}{\cosh\left(\pi\varphi(\la)\right)},
\q f_a(-\la)=f_a(\la),\q 0\leq\la\leq\pi,
\eeq
where $\varphi(\la)=(a/2)\cot\la$ and $a$ is a positive parameter.

Using the technique of orthogonal polynomials on the unit circle
and Szeg\H{o}'s results, Rosenblatt \cite{Ros} proved the following theorem.
\begin{TA}[Rosenblatt \cite{Ros}]
\label{R2}
Suppose that the process $X(t)$ has spectral density $f_a$ given by \eqref{nd4}.
Then the following asymptotic relation for the prediction error $\si^2_n(f_a)$
holds:
\beq
\label{nd6}
\si^2_n(f_a)\sim\frac{\Gamma^2\left(({a+1)}/2\right)}
{\pi 2^{2-a}} \ n^{-a}
\q {\rm as} \q n\to\f.
\eeq
\end{TA}
Note that the function in \eqref{nd4} 
was first considered by Pollaczek \cite{Po}, and then by Szeg\H{o} \cite{S1}, as
a weight-function of a class of orthogonal polynomials that 
serve as illustrations for certain 'irregular' phenomena in the theory of orthogonal polynomials. 
For the function $f_a$ in \eqref{nd4}, we have the following asymptotic
relation (for details see Szeg\H{o} \cite{S1} and Example \ref{ex3} below):
\beq \label{tt2}
f_a(\la)\sim
\left \{
\begin{array}{ll}
 2e^a\exp\left\{-{a\pi}/{|\la|}\right\} & \mbox{as $\la\to0$},\\
2\exp\left\{-{a\pi}/{(\pi-|\la|)}\right\} & \mbox{as $\la\to\pm\pi$}.
\end{array}
\right.
\eeq
Thus, the function $f_a$ in \eqref{nd4} has a very high order of contact
with zero at points $\la=0,\pm\pi$, due to which the process with spectral
density $f_a$ is deterministic and the prediction error $\si^2_n(f_a)$
in \eqref{nd6} decreases to zero like a power as $n\to\f$.

In Babayan et al. \cite{BGT} was proved that if the spectral
density $f$ is such that the sequence $\{\si_n(f)\}$ is weakly varying
(a term defined in Section \ref{WV}) and if, in addition, $g$ is
a nonnegative function that can have polynomial
type singularities, then the sequences $\{\si_n(fg)\}$ and $\{\si_n(f)\}$
have the same asymptotic behavior as $n\to\f$
(see Theorem B in Section \ref{d1}).
Using this result, Rosenblatt's Theorems A was extended in Babayan et al. \cite{BGT} to a
class of spectral densities of the form $f=f_ag$,
where $f_a$ is as in (\ref{nd4}) and $g$ is a nonnegative function that
can have polynomial type singularities
(see Theorem C in Section \ref{d1}).

In this paper, we extend the above quoted results to a broader class of
spectral densities, for which  
the function $g$ can have {\it arbitrary power type singularities}.

Throughout the paper we will use the following notation.\\
The standard symbols $\mathbb{N}$, $\mathbb{Z}$, $\mathbb{R}$ and $\mathbb{C}$
denote the sets of natural, integer, real and complex numbers, respectively.
Also, we denote $\Lambda: = [-\pi, \pi],$ 
$\mathbb{T}:=\{z\in \mathbb{C}: \, |z|=1\}$.
For a point $\la_0\in \Lambda$ and a number $\de>0$ by $O_\de(\la_0)$ we
denote a $\de$-neighborhood of $\la_0$, that is,
$O_\de(\la_0):=\{\la\in \Lambda: \ |\la-\la_0|<\de\}$.
By $L^p(\mu):=L^p(\mathbb{T},\mu)$ ($p\geq $1) we denote the weighted
Lebesgue space with respect to the measure $\mu$, and by $(\cdot,\cdot)_{p,\mu}$
and $||\cdot||_{p,\mu}$ we denote the inner product and the norm in $L^p(\mu)$, respectively.
In the special case where $\mu$ is the Lebesgue measure, 
we will use the notation $L^p$,
$(\cdot,\cdot)_{p}$ and $||\cdot||_{p}$, respectively.
For a function $f\geq0$ by $G(f)$ we denote the geometric mean of $f$.
For two functions $f(\la)\geq0$ and $g(\la)\geq0$, $\la \in \Lambda$,
we will write $f(\lambda){\sim}g(\lambda)$ as ${\lambda\to\lambda_0}$
if $\lim_{\lambda\to\lambda_0}{f(\lambda}/{g(\lambda)}=1$,
and $f(\lambda){\simeq}g(\lambda)$ as ${\lambda\to\lambda_0}$ if
$\lim_{\lambda\to\lambda_0}{f(\lambda}/{g(\lambda)}=c>0$.
We will use similar notation for sequences: for two sequences
$\{a_n\geq0, n\in\mathbb{N}\}$ and $\{b_n>0, n\in\mathbb{N}\}$,
we will write $a_n\sim b_n$ if $\lim_{n\to\f}{a_n}/{b_n}=1$,
$a_n{\simeq} b_n$ if $\lim_{n\to\f}{a_n}/{b_n}=c>0$;
$a_n=O(b_n)$ if ${a_n}/{b_n}$ is bounded,
and $a_n=o(b_n)$ if ${a_n}/{b_n}\to0$ as $n\to\f$.
The letters $C$, $c$, $M$ and $m$ with or without indices are used
to denote positive constants, the values of which can vary from line to line.

The paper is organized as follows.
In Section \ref{pre} we present some necessary notions and preliminary
results that are used throughout the paper.
In Section \ref{d1} we state some auxiliary results.
In Section \ref{d} we state the main results of the paper.
Section \ref{P} contains the proofs of the main results.
In Section \ref{Examples} we discuss some examples illustrating the obtained results.

\s{Preliminaries}
\label{pre}

In this section we present some notions and auxiliary results, which will be used in the sequel:
the Kolmogorov-Szeg\H{o} Theorem, formulas and some properties of the finite prediction error,
properties of the geometric mean of a function, and definition and properties of weakly
varying sequences.

\sn{Kolmogorov-Szeg\H{o}'s Theorem}
Let $X(t)$ be a centered discrete-time
stationary process defined on a probability space
$(\Om, \mathcal{F}, P)$ with covariance function $r(t)$, $t\in \mathbb{Z}$.
By the Herglotz theorem (see, e.g., Brockwell and Davis \cite{BD}, p. 117-18),
there is a finite measure $\mu$ on $\Lambda$ such that
the covariance function $r(t)$
admits the following {\it spectral representation}:
\beq
\label{i1}
r(t)=\inl e^{-it\la}d\mu(\la), \q t\in \mathbb{Z}.
\eeq
\n
The measure $\mu$ in (\ref{i1}) is called the {\it spectral measure} of the
process $X(t)$.
If $\mu$ is absolutely continuous (with respect to the Lebesgue measure),
then the function $f(\la):=d\mu(\la)/d\la$ is called the {\it spectral density}
of $X(t)$. 
We assume that $X(t)$ is a {\it non-degenerate} process, that is,
${\rm Var}[X(0)]:=\E|X(0)|^2=r(0)>0$ and, without loss of generality,
we may take $r(0)=1$.
Also, to avoid the trivial cases, we assume that the spectral measure
$\mu$ is {\it non-trivial}, that is, the support of $\mu$ has positive Lebesgue measure.

\begin{rem}
{\rm The parametrization of the unit circle $\mathbb{T}$ by the formula
$z=e^{i\la}$ establishes a bijection between $\mathbb{T}$ and the
interval $[-\pi,\pi)$. By means of this bijection the measure $\mu$ on
$\Lambda$ generates the corresponding measure on the unit circle $\mathbb{T}$,
which we also denote by $\mu$. Thus, depending
on the context, the measure $\mu$ will be supported either on $\Lambda$ or
on $\mathbb{T}$.
We use the standard Lebesgue decomposition of the measure $\mu$:
\beq
\label{di1}
d\mu(\la)=d\mu_{a}(\la) +d\mu_s(\la) =f(\la)d\la +d\mu_s(\la),
\eeq
where $\mu_a$ is the absolutely continuous part of $\mu$
(with respect to the Lebesgue measure)
and $\mu_s$ is the singular part of $\mu$, which is the sum
of the discrete and continuous singular components of $\mu$.}
\end{rem}

The next result describes the
asymptotic behavior of the prediction error $\si_n^2(\mu)$ for a stationary
process $X(t)$ with spectral measure $\mu$ of the form \eqref{di1}
and gives a spectral characterization of deterministic and nondeterministic processes
(see, e.g., Grenander and Szeg\H{o} \cite{GS}, p. 44).

\vskip2mm
\n
{\sl Kolmogorov-Szeg\H{o}'s Theorem.} {\it Let $X(t)$ be a non-degenerate stationary
process with spectral measure $\mu$ of the form \eqref{di1}.
The following relations hold.
\bea
\label{c013}
\lim_{n\to\f}\si_n^2(\mu)=\lim_{n\to\f}\si_n^2(f)=\si^2(f)=2\pi G(f),
\eea
where $G(f)$ is the {\it geometric mean} of $f$, namely} 
\beq
\label{a2}
G(f): = \left \{
\begin{array}{ll}
\exp\left\{\frac1{2\pi}\inl\ln f(\la)\,d\la \right\} &
\mbox{if \, $\ln f \in {L}^1(\Lambda)$}\\
           0, & \mbox { otherwise.} \qq
           \end{array}
           \right.
\eeq

It is remarkable that the limit in \eqref{c013} is independent of the singular part $\mu_s$.

The condition $\ln f \in {L}^1(\Lambda)$ in \eqref{a2} is equivalent to
the {\it Szeg\H{o} condition}:
\beq
\label{S}
\inl\ln f(\la)\,d\la > -\f
\eeq
(this equivalence follows because $\ln f(\la)\le f(\la)$).
The Szeg\H{o} condition \eqref{S} is also called the {\it non-determinism condition}.

In this paper 
we consider the class of deterministic
processes with absolutely continuous spectra.
We will assume that the corresponding spectral densities,
defined on the segment $[-\pi,\pi]$ (or on the unit circle $\mathbb{T}$),
are extended to $\mathbb{R}$ with period $2\pi$.

Following Rosenblatt \cite{Ros}, we will say that the spectral density $f(\la)$ has a
{\it very high order of contact with zero at a point} $\la_0$ if $f(\la)$
is positive everywhere except for the point $\la_0$, due to which
the Szeg\H{o} condition \eqref{S} is violated.

\sn{Properties of the finite prediction error and geometric mean}

Suppose we have observed the values $X(-n), \ldots, X(-1)$ of a centered,
real-valued stationary process $X(t)$ with spectral measure $\mu$
of the form \eqref{di1}.
The {\it one-step ahead linear prediction problem}
in predicting a random variable $X(0)$ based on the observed values
$X(-n), \ldots, X(-1)$ involves finding constants
$\hat c_k:=\hat c_{k,n}$, $k=1,2,\ldots,n$,
that minimize the one-step ahead prediction error: 
\bea
\label{OS1}
\si_{n}^2(\mu):
=\min_{\{c_k\}}{\E}\left|X(0) - \sum_{k=1}^{n}c_kX(-k)\right|^2
={\E}\left|X(0) - \sum_{k=1}^{n}\hat c_kX(-k)\right|^2.
\eea
Using Kolmogorov's isometric isomorphism $V:\, X(t)\leftrightarrow e^{it\la}$
(see, e.g., Babayan et al. \cite{BGT}),
in view of (\ref{OS1}), for the prediction error $\sigma_n^2(\mu)$ we can write
\bea
\label{OS2}
\sigma_n^2(\mu)=
\min_{\{c_k\}}\left\|1 -\sum_{k=1}^{n}c_ke^{-ik\la}\right\|^2_{2,\mu} 
=\min_{\{q_n\in \mathcal{Q}_n\}}\left\|q_n\right\|^2_{2,\mu}, 
\eea
where $||\cdot||_{2,\mu}$ is the norm in $L^2(\mathbb{T},\mu)$, and
\beq
\label{Q_n}
\mathcal{Q}_n: =\left\{q_n:  q_n(z)=z^{n} + c_1z^{n-1}+ \cdots c_n\right\}
\eeq
is the class of monic polynomials (i.e. with $c_0=1$) of degree $n$.
Thus, the problem of finding $\sigma_n^2(\mu)$ becomes to the problem of finding
the solution of the minimum problem \eqref{OS2}-\eqref{Q_n}.

The polynomial $p_n(z):=p_n(z,\mu)$ which solves the minimum problem \eqref{OS2}-\eqref{Q_n}
is called the {\it optimal polynomial} for $\mu$ in the class $\mathcal{Q}_n$.

The next result by Szeg\H{o} solves the minimum problem \eqref{OS2}-\eqref{Q_n}
(see, e.g., Grenander and Szeg\H{o} \cite{GS}, p. 38).
\begin{pp}
\label{Sz2}
The unique solution of the minimum problem \eqref{OS2}-\eqref{Q_n} is given by
$ p_n(z)=\kappa^{-1}_n\I_n(z)$, and the minimum in \eqref{OS2} is equal to
$||p_n||_{2,\mu}=\kappa^{-2}_n$, where $\I_n(z) = \kappa_nz^n + \cdots +l_n$ is the $n^{th}$
orthogonal polynomial on the unit circle associated with the measure $\mu$,
and $\kappa_n$ is the leading coefficient of $\I_n(z)$.
\end{pp}

Thus, for the prediction error $\sigma_n^2(\mu)$ we have the following formula:
\bea
\label{mm1}
\si^2_n(\mu)=
\min_{\{q_n\in \mathcal{Q}_n\}}\left\|q_n\right\|^2_{2,\mu}
=\left\|p_n(\mu)\right\|^2_{2,\mu}=\left\|\kappa^{-1}_n\I_n(\mu)\right\|^2_{2,\mu}=\kappa^{-2}_n.
\eea
Denote by $D_n=D_n(\mu):=\det[r(t-s), \, t,s =0,1,\ldots n]$ the $n^{th}$
Toeplitz determinant generated by the measure $\mu$,
where $r(t)$ is the covariance function given by \eqref{i1}.
Taking into account that $\kappa^{2}_n={D_{n-1}/D_n}$
(see, e.g., Grenander and Szeg\H{o} \cite{GS}, p. 38),
in view of \eqref{mm1} we obtain the following formula
for the prediction error $\sigma_n^2(\mu)$:
\bea
\label{mm1d}
\si^2_n(\mu)=\frac{D_n(\mu)}{D_{n-1}(\mu)}.
\eea

In what follows we assume that the spectral measure $\mu$ is absolutely
continuous with spectral density $f$, and instead of $\si^2_n(\mu)$, $p_n(\mu)$ and $D_n(\mu)$
we use the notation $\si^2_n(f)$, $p_n(f)$ and $D_n(f)$, respectively.

In the next proposition we list a number of properties of the prediction error $\sigma_n^2(f)$.
\begin{pp}
\label{pp3}
The prediction error $\sigma_n^2(f)$ possesses the following properties.
\begin{itemize}
\item[(a)]
$\sigma_n^2(f)$ is a non-decreasing functional of $f$:
$\sigma_{n}^2(f_1)\leq\sigma_n^2(f_2)$ when $f_1(\lambda)\leq f_2(\lambda)$, $\la\in [-\pi,\pi]$.
\item[(b)]
If $f(\la)=g(\la)$ almost everywhere on $[-\pi,\pi]$, then $\si_n^2(f) = \si_n^2(g).$

\item[(c)]
For any positive constant $c$ we have $\si_n^2(cf) = c\si_n^2(f)$.
\item[(d)]
If $\bar f(\la)=f(\la-\la_0)$, $\la_0\in [-\pi,\pi]$, then $\si_n^2(\bar f) = \si_n^2(f).$
\end{itemize}
\end{pp}
\begin{proof}
To prove assertion (a), observe that by the definition of optimal polynomials
$p_n(z,f_1)$ and $p_n(z,f_2)$, corresponding to spectral densities $f_1$ and $f_2$,
respectively, and formula  \eqref{mm1} we have
\beaa
\sigma_n^2(f_1)=\|p_n(f_1)\|^2_{2,f_1}\leq \|p_n(f_2)\|^2_{2,f_1}
\leq \|p_n(f_2)\|^2_{2,f_2}=\sigma_n^2(f_2).
\eeaa
In the last relation the first inequality follows from the optimality
of the polynomial $p_n(z,f_1)$, while the second inequality follows from
assumption that $f_1(\lambda)\leq f_2(\lambda)$, $\lambda\in \Lambda$.

We show that assertions (b)--(d) follow from formula \eqref{mm1d}.
Indeed, observing that the elements of the Toeplitz determinant $D_n(f)$
being integrals are invariant with respect to null sets, in view of \eqref{mm1d} we obtain assertion (b).
To prove assertion (c), observe that for the Toeplitz determinants generated by
the functions $cf$ and $f$ we have $D_n(cf)= c^nD_n(f)$
(see Grenander and Szeg\H{o} \cite{GS}, p. 64-65).
Hence in view of \eqref{mm1d} we have
$$\si_n^2(cf) = \frac{D_n(cf)}{D_{n-1}(cf)}= \frac{c^nD_n(f)}{c^{n-1}D_{n-1}(f)}= c\si_n^2(f).$$

Finally, the assertion (d) follows from \eqref{mm1d} and the fact that
$D_n(\bar f)= D_n(f)$ (see Grenander and Szeg\H{o} \cite{GS}, p. 64-65).
\end{proof}



Recall that for a function $h\geq0$ by $G(h)$ we denote the geometric mean of $h$
(see formula \eqref{a2}).
In the next proposition we list some properties of the geometric mean $G(h)$
(see Babayan et al. \cite{BGT}).
\begin{pp}
\label{p4.2}
The following assertions hold.
\begin{itemize}
\item[(a)]
Let $c>0$, $\al\in\mathbb{R}$, $f(\la)\geq0$ and $g(\la)\geq0$. Then
\beq
\label{gt1}
G(c)=c, \q G(fg)=G(f)G(g), \q G(f^\al)=G^\al(f).
\eeq
\item[(b)]
$G(f)$ is a non-decreasing functional of $f$: if \ $0\leq f(\la)\leq g(\la)$,
then $0\leq G(f)\leq G(g)$.
In particular, if $0\leq f(\la)\leq 1$, then $0\leq G(f)\leq 1$.
\item[(c)]
If $t(\lambda)$ is a nonnegative trigonometric polynomial, then $G(t^\al)>0$
for $\al\in\mathbb{R}$.
\end{itemize}
\end{pp}

\sn{Weakly varying sequences}
\label{WV}

We recall the notion of {\it weakly varying} sequences and state some of
their properties (see Babayan et al. \cite{BGT})., 
This notion will be used in the specification of the class of deterministic
processes to be considered in this paper.
\begin{den}
\label{kd1}
A sequence of non-zero numbers $\{a_n, \, n \in\mathbb{N}\}$ is said to be
weakly varying if $\lim_{n\to\f}{a_{n+1}}/{a_{n}} =1.$
\end{den}

In the next proposition we list some simple properties of the weakly varying
sequences, which can easily be verified (see Babayan et al. \cite{BGT}).
\begin{pp}
\label{p4.1}
The following assertions hold.
\begin{itemize}
\item[(a)]
If $\{a_n, \, n \in\mathbb{N}\}$ is a weakly varying sequence, then
$\lim_{n\to\f}{a_{n+\nu}}/{a_{n}} =1$ for any $\nu\in \mathbb{N}$.
\item[(b)]
If $\{a_n, \, n \in\mathbb{N}\}$ is a sequence such that $a_n\to a\neq0$
as $n\to\f$, then $\{a_n\}$ is a weakly varying sequence.
\item[(c)]
If $\{a_n, \, n \in\mathbb{N}\}$ and $\{b_n, \, n \in\mathbb{N}, \}$ are
weakly varying sequences, then $ca_n$ $(c\neq0),$
$a_n^\al \, (\al\in\mathbb{R}, a_n>0)$, $a_nb_n$ and $a_n/b_n$ also are
weakly varying sequences.
\item[(d)]
If $\{a_n, \, n \in\mathbb{N}\}$ is a weakly varying sequence,
and $\{b_n, \, n \in\mathbb{N}\}$ is a sequence of non-zero numbers such that
$\lim_{n\to\infty}{b_n}/{a_n}=c\neq 0,$
then $\{b_n, \, n \in\mathbb{N}\}$ is also a weakly varying sequence.
\end{itemize}
\end{pp}

\section{Asymptotic behavior of the prediction error: Auxiliary results}
\label{d1}
In this section we state some auxiliary results concerning asymptotic
behavior of the prediction error (cf. Babayan et al. \cite{BGT}).

In what follows we consider the class of {\it deterministic} processes
possessing spectral densities for which the sequence of
prediction errors $\{\si_n(f)\}$ is weakly varying,
and denote by $\mathcal{F}$ the class of the corresponding spectral densities:
\beq
\label{k04}
\mathcal{F}:=\left\{f\in L^1(\Lambda): \, \, f\geq 0, \, \, G(f)=0, \,\, \lim_{n\to\f} \frac{\sigma_{n+1}(f)}{\sigma_{n}(f)} =1\right\}.
\eeq
\begin{rem}
\label{vvv2}
{\rm According to Rakhmanov's theorem
(see Rakhmanov \cite{Ra4} 
and Babayan et al. \cite{BGT})
a sufficient condition for $f\in\mathcal{F}$ is that 
$f>0$ almost everywhere on $\Lambda$.
Thus, the considered class $\mathcal{F}$ includes all the almost everywhere positive spectral densities.
On the other hand, according to Theorem 3.2 and Remark 3.7 of
Babayan et al. \cite{BGT}, the class $\mathcal{F}$ does not contain spectral densities, which vanish on an entire segment of $\Lambda$
(or on an arc of the unit circle $\mathbb{T}$).}
\end{rem}

\begin{den}
\label{kd33}
Let $\mathcal{F}$ be the class of spectral densities defined by \eqref{k04}.
For $f\in \mathcal{F}$ denote by $\mathcal{M}_f$ the class of 
nonnegative functions $g(\la)$ $(\la\in\Lambda)$ satisfying the following
three conditions:
$G(g)>0$, $fg\in L^1(\Lambda)$, and
\beq
\label{e5.6} 
\lim_{n\to\f}\frac{\si^2_{n}(fg)}{\si^2_{n}(f)} =G(g), 
\eeq
that is,
\beq
\label{e5.6a}
\mathcal{M}_f:=\left\{g\geq 0, \, \, G(g)>0, \,\, fg\in L^1(\Lambda),
\, \, \lim_{n\to\f}\frac{\si^2_{n}(fg)}{\si^2_{n}(f)} =G(g)\right\}.
\eeq
\end{den}


The next proposition shows that the class $\mathcal{F}$
is close under multiplication by functions from the class $\mathcal{M}_f$.
\begin{pp}
\label{ppg1}
If $f\in \mathcal{F}$ and $g\in \mathcal{M}_f$, then $fg\in \mathcal{F}$.
\end{pp}
\begin{proof}
According to the definition of the class $\mathcal{F}$, we have
to show that $fg\in L^1(\Lambda)$, $G(fg)=0$, and
\beq
\label{gg}
\lim_{n\to\f}\frac{\si_{n+1}(fg)}{\si_{n}(fg)} =1.
\eeq
The assertion $fg\in L^1(\Lambda)$ follows from the condition
$g\in \mathcal{M}_f$, while $G(fg)=0$ follows from the condition
$f\in \mathcal{F}$ and Proposition \ref{p4.2}(a): $G(fg)=G(f)G(g)=0$.
As for the relation \eqref{gg}, 
we can write
\beaa
&&\lim_{n\to\infty}\frac{\sigma_{n+1}(fg)}{\sigma_{n}(fg)}
=\lim_{n\to\infty}\frac{\sigma_{n+1}(fg)}{\sigma_{n+1}(f)}
\cd\frac{\sigma_{n+1}(f)}{\sigma_{n}(f)}
\cd\frac{\sigma_{n}(f)}{\sigma_{n}(fg)}\\
&&=\lim_{n\to\infty}\frac{\sigma_{n+1}(fg)}{\sigma_{n+1}(f)}
\cd \lim_{n\to\infty}\frac{\sigma_{n+1}(f)}{\sigma_{n}(f)}
\cd \lim_{n\to\infty}\frac{\sigma_{n}(f)}{\sigma_{n}(fg)}
= \sqrt{G(g)}\cd 1/\sqrt{G(g)}=1,
\eeaa
and the result follows.
\end{proof}
The next result shows that the class $\mathcal{M}_f$ in a certain sense
is close under multiplication.
\begin{pp}
\label{ppg}
Let $f\in \mathcal{F}$. If $g_1\in \mathcal{M}_f$ and $g_2\in \mathcal{M}_{fg_1}$, then $g:=g_1g_2\in \mathcal{M}_f$ and $fg\in \mathcal{F}$.
In particular, if $g\in \mathcal{M}_f\cap\mathcal{M}_{fg}$,
then $g^2\in \mathcal{M}_f$. 
\end{pp}
\begin{proof}
By the definition of the classes $\mathcal{M}_f$ and $\mathcal{M}_{fg_1}$, we have $G(g_1)>0$, $G(g_2)>0$, and hence, by Proposition \ref{p4.2}(a),
$G(g)>0$. By Proposition \ref{ppg1} we have $fg_1\in \mathcal{F}$. Hence, taking into account that $g_2\in \mathcal{M}_{fg_1}$ we have $fg_1g_2\in L^1(\Lambda)$, and
\beq
\label{gg1}
\lim_{n\to\f}\frac{\si^2_{n}(fg_1g_2)}{\si^2_{n}(fg_1)} =G(g_2).
\eeq
Then, we can write
\beaa
&&\lim_{n\to\infty}\frac{\sigma_n^2(fg)}{\sigma_n^2(f)}
=\lim_{n\to\infty}\frac{\sigma_n^2(fg_1g_2)}{\sigma_n^2(f)}
=\lim_{n\to\infty}\frac{\sigma_n^2(fg_1g_2)}{\sigma_n^2(fg_1)}
\cd \frac{\sigma_n^2(fg_1)}{\sigma_n^2(f)}\\
&&=\lim_{n\to\infty}\frac{\sigma_n^2(fg_1g_2)}{\sigma_n^2(fg_1)}
\cd \lim_{n\to\infty}\frac{\sigma_n^2(fg_1)}{\sigma_n^2(f)}
= G(g_1)\cd G(g_2)=G(g).
\eeaa
In the last relation the fourth equality follows from \eqref{gg1}
and the condition $g_1\in \mathcal{M}_f$, while the fifth equality
follows from Proposition \ref{p4.2}(a).
Thus, we have proved that $g\in \mathcal{M}_f$, from this and
Proposition \ref{ppg1} it follows that $fg\in \mathcal{F}$.
\end{proof}

In the next definition we introduce certain classes of bounded functions.
\begin{den}
\label{kd2}
We define the class $B$ to be the set of all nonnegative, Riemann integrable on
$\Lambda=[-\pi,\pi]$ functions $h(\lambda)$. Also, we define the following subclasses:
\beq
\nonumber 
B_+:= \{h \in B: \, h(\lambda)\geqslant m\},\q
B^-:= \{h \in B: \, h(\lambda)\leqslant M\}, \q B_+^-:=B_+\cap B^-,
\eeq
where $m$ and $M$ are some positive constants.
\end{den}

In the next proposition we list some obvious properties of the classes
$B_+$, $B^-$ and $B_+^-$.
\begin{pp}
\label{p3.3}
The following assertions hold.
\begin{itemize}
\item[a)] If $h\in B_+ (B^-)$, then $1/h\in B^- (B_+)$.

\item[b)] If $h_1, h_2 \in B_+ (B^-)$, then $h_1+ h_2 \in B_+ (B^-)$ and
$h_1 h_2 \in B_+ (B^-)$.

\item[c)] If $h_1, h_2 \in B^-$ and $h_1/h_2$ is bounded,
then $h_1/h_2 \in B^-$.

\item[d)] If $h_1, h_2 \in B_+^-$, then $h_1+ h_2 \in B_+^-$,
 $h_1 h_2 \in B_+^-$ and $h_1/h_2 \in B_+^-$.
\end{itemize}
\end{pp}

The following theorem, proved in Babayan et al. \cite{BGT},
describes the asymptotic behavior of the ratio $\sigma_n^2(fg)/\sigma_n^2(f)$
as $n\to\f$, and essentially states that if the spectral density $f$
is from the class $\mathcal{F}$ (see \eqref{k04}), and $g$
is a nonnegative function, 
which can have {\it polynomial} type singularities, 
then the sequences $\{\si_n(fg)\}$ and
$\{\si_n(f)\}$  have the same asymptotic behavior as $n\to\f$ up to a
positive numerical factor.

\begin{TB}[Babayan et al. \cite{BGT}]
Let $f$ be an arbitrary function from the class $\mathcal{F}$, and let
$g$ be a function of the form:
\beq
\label{g}
g(\la)=h(\la)\cd\frac{t_1(\la)}{t_2(\la)}, \q \la\in\Lambda,
\eeq
where  $h\in B_+^-$,  $t_1$ and $t_2$ are nonnegative trigonometric polynomials, such that $fg\in L^1(\Lambda)$.
Then $g\in\mathcal{M}_f$ and $fg\in\mathcal{F}$, that is, $fg$ is the spectral density of a deterministic process with weakly varying prediction error,
and the relation \eqref{e5.6} holds.
\end{TB}


Taking into account that the sequence $\{n^{-\al}, \, \, n\in\mathbb{N}, \, \al>0\}$
is weakly varying, as an immediate consequence of Theorem B, we have the following result.
\begin{cor}[Babayan et al. \cite{BGT}]
\label{c4.2}
Let the functions $f$ and $g$ be as in Theorem B,
and let $\sigma_n(f)\sim cn^{-\al}$ ($c>0, \al>0$) as $n\to\f$. Then
\beq
\label{s88}
\nonumber
\sigma_n(fg)\sim c G(g)n^{-\al} \q {\rm as} \q n\to\f,
\eeq
where $G(g)$ is the geometric mean of $g$.
\end{cor}
\label{c4.3}
The next result, which immediately follows from Theorem B and Corollary \ref{c4.2},
extends Rosenblatt's Theorem A.
\begin{TC}[Babayan et al. \cite{BGT}]
Let $f=f_ag$, where $f_a$ is defined by (\ref{nd4})
and $g$ satisfies the assumptions of Theorem B. Then
\beq
\label{s8}
\nonumber
 \si^2_n(f)\sim\frac{\Gamma^2\left(\frac{a+1}2\right)G(g)}
{\pi 2^{2-a}} \ n^{-a} 
\q {\rm as} \q n\to\f,
\eeq
where $G(g)$ is the geometric mean of $g$.
\end{TC}
We thus have the same limiting behavior for $\si^2_n(f)$ as in
the Rosenblatt's relation \eqref{nd6} up to an additional
positive factor $G(g)$.

\section{Asymptotic behavior of the prediction error: The main results}
\label{d}
In this section we state the main results of this paper,
extending the above stated Theorems B and C to a broader class of
spectral densities, for which the function $g$ can have
{\it arbitrary power type singularities}.

\begin{thm}
\label{T1}
Let $f$ be an arbitrary function from the class $\mathcal{F}$, and let
$g$ be a function of the form:
\beq
\label{e5.5}
g(\la)=h(\la)\cd |t(\la)|^\al, \q\al>0, \, \, \la\in\Lambda,
\eeq
where $h\in B_+^-$ and $t$ is an arbitrary trigonometric polynomial.
Then $g\in\mathcal{M}_f$ and $fg\in\mathcal{F}$, that is, $fg$ is the spectral density of a deterministic process with weakly varying prediction error,
and the relation \eqref{e5.6} holds.
\end{thm}
\begin{cor}
\label{ec5.2}
The conclusion of Theorem \ref{T1} remains valid if the function  $g$
has the following form:
\beq
\nonumber 
g(\la)=h(\la)\cd |t_1(\la)|^{\al_1}\cd |t_2(\la)|^{\al_2}\cd\cdots\cd
|t_m(\la)|^{\al_m}, \q \la\in\Lambda,
\eeq
where  $h\in B_+^-$, \, $t_1, t_2, \ldots, t_m$ are arbitrary
trigonometric polynomials, $\al_1, \al_2, \ldots \al_m$ are arbitrary
positive numbers, and $m\in\mathbb{N}$.
\end{cor}
\begin{thm}
\label{T2}
Let $f$ be an arbitrary function from the class $\mathcal{F}$, and let
$g$ be a function of the form:
\beq
\label{e5.5a}
g(\la)=h(\la)\cd t^{-\al}(\la), \q \al>0, \, \, \la\in\Lambda,
\eeq
where  $h\in B_+^-$ and $t$ is a nonnegative trigonometric polynomial.
Then $g\in\mathcal{M}_f$ and $fg\in\mathcal{F}$ provided that
$ft^{-(k+1)}\in L^1(\Lambda)$, where $k:=[\al]$ is the integer part of $\al$. 
\end{thm}

To state the next result we need the following definition.
\begin{den}
\label{kd3}
Let $E_1$ and $E_2$ be two numerical sets such that for any $x\in E_1$ and $y\in E_2$ we have $x<y$. We say that the sets $E_1$ and $E_2$ are separated from each other
if $\sup E_1<\inf E_2.$ Also, we say that a numerical set $E$ is separated from
infinity if it is bounded from above.
\end{den}
\begin{thm}
\label{T2c}
Let $f(\la)$ and $\hat f(\la)$ $(\la\in\Lambda)$ be spectral densities
of stationary processes satisfying the following conditions:
\begin{itemize}
\item[1)] $f, \hat f\in B^-$;

\item[2)] the functions  $f(\la)$ and $\hat f(\la)$ have $k$ common essential
zeros $\la_1,\la_2,\ldots,\la_k\in\Lambda$ $(-\pi<\la_1<\la_2<\cdots<\la_k\leq \pi, \, k\in\mathbb{N})$,
that is,
\beq
\label{s3.7}
\lim_{\la\to\la_j}f(\la)=\lim_{\la\to\la_j}\hat f(\la)=0, \q j=1,2,\ldots,k;
\eeq
\item[3)] the functions  $f(\la)$ and $\hat f(\la)$ are infinitesimal of
the same order in a neighborhood of each point $\la_j$ $(j=1,2,\ldots,k)$, that is,
\beq
\label{s3.8}
\lim_{\la\to\la_j}\frac{\hat f(\la)}{f(\la)}=c_j>0, \q j=1,2,\ldots,k;
\eeq
\item[4)] the functions  $f(\la)$ and $\hat f(\la)$ are bounded away from zero
outside any neighborhood $O_\de(\la_j)$ $(j=1,2,\ldots,k)$, which is separated
from the neighboring zeros $\la_{j-1}$ and $\la_{j+1}$ of $\la_j$,
that is, there is a number $m:=m_\de>0$ such that $f(\la)\geq m$ and $\hat f(\la)\geq m$
for almost all $\la\notin\cup_{j=1}^kO_\de(\la_j)$.
Then the following assertions hold:
\end{itemize}
\begin{itemize}
\item[a)] $h(\la):=\frac{\hat f(\la)}{f(\la)}\in B_+^-;$
\item[b)] the processes with spectral densities $f$ and $\hat f$ either
both are deterministic or both are nondeterministic;
\item[c)] if one of the functions $f$ and $\hat f$ is from the class $\mathcal{F}$, then so is the other, and the following relation holds:
\beq
\label{g3.6e}
\lim_{n\to\f}\frac{\si_n^2(\hat f)}{\si_n^2(f)}=G(h)>0.
\eeq
\end{itemize}
\end{thm}

\begin{rem}
\label{r3.2}
{\rm The conditions of Theorem \ref{T2c} mean that the points $\la_j$
$(j=1,2,\ldots,k)$ are the only common zeros of functions $f(\la)$ and
$\hat f(\lambda)$. Besides, in the case of deterministic processes,
at least one of these zeros should be of sufficiently high order.
Also, notice that the conditions 1) and 4) of Theorem \ref{T2c} will be satisfied
if the functions $f(\la)$ and $\hat f(\lambda)$ are continuous on $\Lambda$.}
\end{rem}


\begin{thm}
\label{T1b}
Let $f$ be an arbitrary function from the class $\mathcal{F}$, and let
$g$ be a function of the form:
\beq
\label{e5.5b}
g(\la)=h(\la)\cd |q(\la)|^\al, \q \al\in\mathbb{R}, \, \, \la\in\Lambda,
\eeq
where  $h\in B_+^-$, $q$ is an arbitrary algebraic polynomial
with real coefficients, and $fg\in L^1(\Lambda)$.
Then $fg\in\mathcal{F}$ and $g\in\mathcal{M}_f$. 
\end{thm}

The next result extends Theorem C to a broader class of functions $g$.
\begin{cor}
\label{c5.5}
Let $f=f_ag$, where $f_a$ is defined by (\ref{nd4}),
and let $g$ be a function satisfying the conditions of one of Theorems \ref{T1}, \ref{T2}, \ref{T1b} or Corollary \ref{ec5.2}. Then
\beq
\label{s888}
\nonumber
\de_n(f) = \si^2_n(f)\sim\frac{\Gamma^2\left(\frac{a+1}2\right)G(g)}
{\pi 2^{2-a}} \ n^{-a} 
\q {\rm as} \q n\to\f,
\eeq
where $G(g)$ is the geometric mean of $g$.
\end{cor}

\begin{rem}
\label{vvv3}
{\rm In view of Remark \ref{vvv2} it follows that all the above stated results
remain true if the condition $f\in\mathcal{F}$
is replaced by the slightly strong but more constructive condition:
'{\it the spectral density $f$ is positive ($f>0$) almost everywhere on $\Lambda$}'.}
\end{rem}

\s{Proofs}
\label{P}
In this section we prove the main results of this paper stated in Section \ref{d}.
We first establish a number of lemmas.

\sn{Lemmas}
\label{Lem}

\begin{lem}
\label{ekl2}
Let $f$ be an arbitrary function from the class $\mathcal{F}$, and let
$t$ be a nonnegative trigonometric polynomial. Then
$t^{1/2}\in\mathcal{M}_f$ and $ft^{1/2}\in\mathcal{F}$.
\end{lem}
\begin{proof}
Observe first that, by Proposition \ref{ppg1}, the second assertion
of the lemma ($ft^{1/2}\in\mathcal{F}$) follows from the first
assertion ($t^{1/2}\in\mathcal{M}_f$). So, to complete the proof of the lemma we have to prove
the relation $t^{1/2}\in\mathcal{M}_f$.
To this end, we first verify the first two conditions for
$t^{1/2}$ to belong to the class $\mathcal{M}_f$, that is, the
conditions: $G(t^{1/2})>0$ and $ft^{1/2}\in L^1(\Lambda)$ (see \eqref{e5.6a}).
First, the condition $G(t^{1/2})>0$ follows from Proposition \ref{p4.2}(c).
Next, observing that the function $t^{1/2}$ is continuous, and hence
is bounded, in view of the condition $f\in\mathcal{F}$, we conclude that $ft^{1/2}\in L^1(\Lambda)$.
Also, observe that by Proposition \ref{p4.2}(a), we have
$G(ft^{1/2})=G(f)G(t^{1/2})=0$, showing that $ft^{1/2}$ is the spectral
density of a deterministic process.

Now we proceed to verify the third condition for $t^{1/2}$
to belong to the class $\mathcal{M}_f$, that is, the relation:
\bea
\label{e5.7}
\lim_{n\to\infty}\frac{\sigma_n^2(ft^{1/2})}{\sigma_n^2(f)}= G(t^{1/2})>0.
\eea

To do this, observe first that since $t^{1/2}\in B^-$, we can apply Lemma 4.5
of Babayan et al. \cite{BGT} to get
\beq
\label{e5.8}
\limsup_{n\to\infty}\frac{\sigma_n^2(ft^{1/2})}{\sigma_n^2(f)}\leq G(t^{1/2}).
\eeq

So, we have to prove the inverse inequality:
\beq
\label{e5.9}
\liminf_{n\to\infty}\frac{\sigma_n^2(ft^{1/2})}{\sigma_n^2(f)}\geq G(t^{1/2})>0.
\eeq
To do this, we consider the function:
\beq
\label{e5.10}
\hat f(\lambda):=f(\lambda)t(\lambda), \q \la\in\Lambda.
\eeq
Applying Theorem B with $h(\la)\equiv t_2(\la)\equiv1$,
$t_1(\la)=t(\la)$ we conclude that $t\in\mathcal{M}_f$ and $\hat f\in\mathcal{F}$, and, in particular, 
\beq
\label{kk56}
\lim_{n\to\f}\frac{\si^2_{n}(\hat f)}{\si^2_{n}(f)} =G(t)>0.
\eeq
Next, taking into account that $t^{-1/2}\in B_+$, we can apply Lemma 4.6
of Babayan et al. \cite{BGT} to obtain
\beq
\label{e5.11}
\liminf_{n\to\infty}\frac{\sigma_n^2(\hat f\cd t^{-1/2})}{\sigma_n^2(\hat f)}
\geq G(t^{-1/2})>0.
\eeq
Now we can write
\bea
\label{e5.12}
\nonumber
&&\liminf_{n\to\infty}\frac{\sigma_n^2(ft^{1/2})}{\sigma_n^2(f)}
=\liminf_{n\to\infty}\frac{\sigma_n^2(ft\cd t^{-1/2})}{\sigma_n^2(f)}
=\liminf_{n\to\infty}\frac{\sigma_n^2(ft\cd t^{-1/2})}{\sigma_n^2(ft)}
\cd \frac{\sigma_n^2(ft)}{\sigma_n^2(f)}\\
&&=\liminf_{n\to\infty}\frac{\sigma_n^2(\hat f\cd t^{-1/2})}{\sigma_n^2(\hat f)}
\cd \liminf_{n\to\infty}\frac{\sigma_n^2(\hat f)}{\sigma_n^2(f)}
\geq G(t^{-1/2})\cd G(t)=G(t^{1/2}).
\eea
In the last relation, the third equality follows from \eqref{e5.10},
the inequality after that follows from \eqref{kk56} and \eqref{e5.11},
and the last equality follows from Proposition \ref{p4.2}(a).
Thus, the inequality \eqref{e5.9} is proved.

Combining the inequalities
\eqref{e5.8} and \eqref{e5.9} we obtain the desired equality \eqref{e5.7},
which together with the above obtained facts: $G(t^{1/2})>0$ and
$ft^{1/2}\in L^1(\Lambda)$ means that $t^{1/2}\in\mathcal{M}_f$.
Lemma \ref{ekl2} is proved.
\end{proof}
\begin{lem}
\label{ekl3}
Let $f$ be an arbitrary function from the class $\mathcal{F}$, and let
$t$ be a nonnegative trigonometric polynomial.
Then $t^{1/2^m}\in\mathcal{M}_f$ and $ft^{1/2^m}\in\mathcal{F}$
for any $m\in\mathbb{N}$. 
\end{lem}
\begin{proof}
As in the proof of Lemma \ref{ekl2}, we have only to prove the relation
$t^{1/2^m}\in\mathcal{M}_f$.
To this end, similar to Lemma \ref{ekl2}, we first verify the first two conditions for $t^{1/2^m}$ to belong to the class $\mathcal{M}_f$, that is,
$G(t^{1/2^m})>0$ and $ft^{1/2^m}\in L^1(\Lambda)$.
Also, we observe that $G(ft^{1/2^m})=0$, showing that $ft^{1/2^m}$
is the spectral density of a deterministic process.

Next, we use induction on $m$ to prove the third condition for $t^{1/2^m}$
to belong to the class $\mathcal{M}_f$, that is, the relation:
\bea
\label{e5.13}
\lim_{n\to\infty}\frac{\sigma_n^2(f t^{1/2^m})}{\sigma_n^2(f)}= G(t^{1/2^m})>0.
\eea

Observe first that for $m=1$ the relation \eqref{e5.13} coincides with \eqref{e5.7}.
Now assuming that it is satisfied for $m=\nu$, that is,
\beq
\label{e5.14}
\lim_{n\to\infty}\frac{\sigma_n^2(f t^{1/2^\nu})}{\sigma_n^2(f)}= G(t^{1/2^\nu})>0,
\eeq
we prove that it remains valid for $m=\nu+1$, that is,
\beq
\label{e5.15}
\lim_{n\to\infty}\frac{\sigma_n^2(f t^{1/2^{(\nu+1)}})}{\sigma_n^2(f)}
= G(t^{1/2^{(\nu+1)}})>0.
\eeq
To this end, observe first that arguments similar to those used in the proof
of Lemma \ref{ekl2} can be applied to obtain (cf. \eqref{e5.8}):
\beq
\label{e5.16}
\limsup_{n\to\infty}\frac{\sigma_n^2(f t^{1/2^{(\nu+1)}})}{\sigma_n^2(f)}
\leq G(t^{1/2^{(\nu+1)}}).
\eeq
The proof of the inverse inequality
\beq
\label{e5.17}
\liminf_{n\to\infty}\frac{\sigma_n^2(f t^{1/2^{(\nu+1)}})}{\sigma_n^2(f)}
\geq G(t^{1/2^{(\nu+1)}})
\eeq
is similar to that of the inequality \eqref{e5.9}.
Indeed, observe that  $t^{-1/2^{(\nu+1)}}\in B_+$, and by the inductive assumption the function
\beq
\label{e5.18}
\hat f(\lambda):=f(\lambda)t^{1/2^\nu}(\la),\q \la\in\Lambda
\eeq
belongs to the class $\mathcal{F}$.
Hence, we can apply Lemma 4.6 of Babayan et al. \cite{BGT} to obtain
\beq
\label{e5.19}
\liminf_{n\to\infty}\frac{\sigma_n^2(\hat f t^{-1/2^{(\nu+1)}})}{\sigma_n^2(\hat f)}
\geq G(t^{-1/2^{(\nu+1)}})>0.
\eeq
Next, we can write
\bea
\nonumber
&&\liminf_{n\to\infty}\frac{\sigma_n^2(f t^{1/2^{(\nu+1)}})}{\sigma_n^2(f)}
=\liminf_{n\to\infty}\frac{\sigma_n^2(f t^{1/2^\nu}\cd t^{-1/2^{(\nu+1)}})}{\sigma_n^2(f)}\\
\nonumber
&&=\liminf_{n\to\infty}\frac{\sigma_n^2(f t^{1/2^\nu}\cd t^{-1/2^{(\nu+1)}})}{\sigma_n^2(f t^{1/2^\nu})}
\cd \frac{\sigma_n^2(f t^{1/2^\nu})}{\sigma_n^2(f)}\\
\nonumber
&&=\liminf_{n\to\infty}\frac{\sigma_n^2(\hat f \cd t^{-1/2^{(\nu+1)}})}{\sigma_n^2(\hat f)}
\cd \liminf_{n\to\infty}\frac{\sigma_n^2(f t^{1/2^\nu})}{\sigma_n^2(f)}\\
&&\geq G(t^{-1/2^{(\nu+1)}})\cd G(t^{1/2^\nu})=G(t^{1/2^{(\nu+1)}}).
\eea
In the last relation, the third equality follows from \eqref{e5.18},
the inequality after that follows from \eqref{e5.19} and inductive assumption \eqref{e5.14},
and the last equality follows from Proposition \ref{p4.2}(a).
Thus, the inequality \eqref{e5.17} is proved.

Combining the inequalities
\eqref{e5.16} and \eqref{e5.17} we obtain the desired equality \eqref{e5.15},
which together with the above obtained facts: $G(t^{1/2^{(\nu+1)}}>0$ and $ft^{1/2^{(\nu+1)}}\in L^1(\Lambda)$, means that $t^{1/2^{(\nu+1)}}\in\mathcal{M}_f$.
Lemma \ref{ekl3} is proved.
\end{proof}

\begin{lem}
\label{ekl4}
Let $f$ be an arbitrary function from the class $\mathcal{F}$, and let
$t$ be a nonnegative trigonometric polynomial.
Then $t^{k/2^m}\in\mathcal{M}_f$ and $ft^{k/2^m}\in\mathcal{F}$ for any $k,m\in\mathbb{N}$.
\end{lem}
\begin{proof}
We use induction on $k$, 
and observe first that
for $k=1$, the assertion of the lemma coincides with that of Lemma \ref{ekl3}.
Now assuming that it is satisfied for $k=\nu$, that is,
$t^{\nu/2^m}\in\mathcal{M}_f$ and $ft^{\nu/2^m}\in\mathcal{F}$,
we prove it for $k=\nu+1$, that is,
$t^{(\nu+1)/2^m}\in\mathcal{M}_f$ and $ft^{(\nu+1)/2^m}\in\mathcal{F}$.

To this end, we set $g_1:=t^{\nu/2^m}$ and $g_2:=t^{1/2^m}$, and
observe that by inductive assumption, we have
$g_1\in\mathcal{M}_f$ and $fg_1\in\mathcal{F}$.
Taking into account the last relation, we can apply Lemma \ref{ekl3}
with $fg_1$ as $f$, and conclude that $g_2\in\mathcal{M}_{fg_1}$.

Therefore, by Proposition \ref{ppg}, we have
$g_1g_2=t^{(\nu+1)/2^m}\in\mathcal{M}_f$ and $fg_1g_2=ft^{(\nu+1)/2^m}\in\mathcal{F}$.
Lemma \ref{ekl4} is proved.
\end{proof}

\begin{lem}
\label{l5-0}
Let $f$ be an arbitrary function from the class $\mathcal{F}$, and let
$t$ be a nonnegative trigonometric polynomial. Then
$t^{-k}\in\mathcal{M}_f$ and $ft^{-k}\in\mathcal{F}$ for any $k\in\mathbb{N}$, provided that $ft^{-k}\in L^1(\Lambda)$. In particular, we have
\bea
\label{e5.02}
\lim_{n\to\infty}\frac{\sigma_n^2(ft^{-k})}{\sigma_n^2(f)}
= G(t^{-k}) >0.
\eea
\end{lem}
\begin{proof}
We use induction on $k$, 
and show that the result follows from Theorem B
by using Proposition \ref{ppg}.
Observe first that for $k=1$ the assertion of the lemma coincides
with Theorem B applied to $h(\la)\equiv t_1(\la)\equiv1$ and $t_2(\la)=t(\la)$.
Now assuming that it is satisfied for $k=\nu$, that is,
$t^{-\nu}\in\mathcal{M}_f$ and $ft^{-\nu}\in\mathcal{F}$
provided that $ft^{-\nu}\in L^1(\Lambda)$, we prove it for $k=\nu+1$,
that is,
$t^{-(\nu+1)}\in\mathcal{M}_f$ and $ft^{-(\nu+1)}\in\mathcal{F}$
provided that $ft^{-(\nu+1)}\in L^1(\Lambda)$.

To this end, we set $g_1(\la):=t^{-\nu}(\la)$, $g_2(\la)=t^{-1}(\la)$
and observe that by inductive assumption, we have
$g_1\in\mathcal{M}_f$ and $fg_1\in\mathcal{F}$.
We show that $g_2\in\mathcal{M}_{fg_1}$. Indeed,
by Proposition \ref{p4.2}(c) we have  $G(g_2)=G(t^{-1})>0$.
Besides, $fg_1g_2=ft^{-(\nu+1)}\in L^1(\Lambda)$ by assumption.
The obtained relations allow to apply Theorem B
with $fg_1$ as $f$ and $g_2$ as $g$,
and conclude that $g_2\in\mathcal{M}_{fg_1}$.

Thus, the functions
$g_1$ and $g_2$ satisfy Proposition \ref{ppg}, and hence
$g_1g_2=t^{-(\nu+1)}\in\mathcal{M}_f$ and $fg_1g_2=ft^{-(\nu+1)}\in\mathcal{F}$.

Lemma \ref{l5-0} is proved.
\end{proof}

\sn{Proof of theorems}
\label{Thm}

\begin{proof}[Proof of Theorem \ref{T1}]
We first verify the first two conditions for a function $g(\lambda)$
to belong to the class $\mathcal{M}_f$, that is, the conditions:
$G(g)>0$ and $fg\in L^1(\Lambda)$.
From the condition $h\in B_+^-$ it follows that $G(h)>0$, while
by Proposition  \ref{p4.2}(c) we have $G(|t|^\al)>0$. Therefore, $G(g)=G(h)G(|t|^\al)>0$.
Since both functions $h(\la)$ and $|t(\la)|^\al$ are bounded,
the function $g(\la)=h(\la)|t(\la)|^\al$ is also bounded, and
$fg\in L^1(\Lambda)$. Besides, we have $G(fg)=G(f)G(g)=0,$
showing that $fg$ is the spectral density of a deterministic process.

Taking into account Proposition \ref{ppg1}, 
to complete the proof of the theorem it remains to verify the relation \eqref{e5.6}.
The proof of relation \eqref{e5.6} we split into four steps.

\noindent
{\sl Step 1.} We prove the relation \eqref{e5.6} in the special case
where $h(\lambda)\equiv 1$ and $t(\lambda)$ is a nonnegative trigonometric polynomial satisfying the condition:
\beq
\label{e4.31}
t(\lambda)\leq 1, \q \lambda\in\Lambda.
\eeq
For an arbitrary natural number $m$ by $k_m$ we denote the integer part of the number
$2^m\cd\al$, that is, $k_m:=[2^m\cd\al]$. Then we have the following inequality:
$$
k_m\leq 2^m\cd\al < k_m+1,
$$
or, equivalently
\beq
\label{e5.26}
\frac{k_m}{2^m}\leq\al< \frac{k_m+1}{2^m}.
\eeq
In view of \eqref{e4.31} and \eqref{e5.26} we can write
\beq
\label{e5.27a}
t^{(k_m+1)/2^m}(\lambda)< t^\al(\lambda)\leq t^{k_m/2^m}(\lambda), \q \lambda\in\Lambda,
\eeq
implying that
\beq
\label{e5.27}
f(\lambda)t^{(k_m+1)/2^m}(\lambda)< f(\lambda)t^\al(\lambda)
\leq f(\lambda)t^{k_m/2^m}(\lambda), \q \lambda\in\Lambda.
\eeq
Taking into account that by Proposition \ref{pp3}(a) the prediction error
$\sigma_n^2(f)$ is a non-decreasing functional of $f$ from \eqref{e5.27}, we obtain
\beq
\nonumber
\sigma_n^2\left(ft^{(k_m+1)/2^m}\right)\leq \sigma_n^2\left(ft^\al\right)\leq \sigma_n^2\left(ft^{k_m/2^m}\right).
\eeq
Dividing the last inequality by $\sigma_n^2(f)$ and passing to the limit as $n\to\f$,
we obtain
\beq
\label{e5.28}
\lim_{n\to\infty}\frac{\sigma_n^2\left(ft^{(k_m+1)/2^m}\right)}{\sigma_n^2(f)}
\leq \liminf_{n\to\infty}\frac{\sigma_n^2\left(ft^\al\right)}{\sigma_n^2(f)}
\leq \limsup_{n\to\infty}\frac{\sigma_n^2\left(ft^\al\right)}{\sigma_n^2(f)}
\leq \lim_{n\to\infty}\frac{\sigma_n^2\left(ft^{k_m/2^m}\right)}{\sigma_n^2(f)}.
\eeq
By Lemma \ref{ekl4}, the first and the last limits in \eqref{e5.28} are equal to
$G\left(t^{(k_m+1)/2^m}\right)$ and $G\left(t^{k_m/2^m}\right)$, respectively.
Hence \eqref{e5.28} can be written as follows:
\beq
\label{e5.28a}
G\left(t^{(k_m+1)/2^m}\right)
\leq \liminf_{n\to\infty}\frac{\sigma_n^2\left(ft^\al\right)}{\sigma_n^2(f)}
\leq \limsup_{n\to\infty}\frac{\sigma_n^2\left(ft^\al\right)}{\sigma_n^2(f)}
\leq G\left(t^{k_m/2^m}\right).
\eeq
On the other hand, taking into account that the geometric mean $G(f)$ is a
non-decreasing functional of $f$ (see Proposition \ref{p4.2}(b)),
from \eqref{e5.27a} we obtain
\beq
\label{e5.29}
G\left(t^{(k_m+1)/2^m}\right)\leq G\left(t^\al\right)
\leq G\left(t^{k_m/2^m}\right).
\eeq
Thus, to complete the proof in the considered case, it remains to show that
for large enough $m$ the quantities $G\left(t^{(k_m+1)/2^m}\right)$ and
$G\left(t^{k_m/2^m}\right)$ are arbitrarily close.
To do this, observe that in view of \eqref{e4.31}
and the properties of the geometric mean (see Proposition \ref{p4.2}(a),(b)),
we can write
\beq
\label{e5.30}
1\leq \frac{G\left(t^{k_m/2^m}\right)}{G\left(t^{(k_m+1)/2^m}\right)} =G\left(t^{k_m/2^m-(k_m+1)/2^m}\right)= G\left(t^{-1/2^m}\right)
=G^{-1/2^m}\left(t\right).
\eeq
Therefore, in view of the limiting relation $\lim_{n\to\infty}c^{1/n}=1$ ($c>0$),
it follows that
\beq
\label{e5.31}
\lim_{m\to\infty}G^{-1/2^m}\left(t\right)=1.
\eeq
Finally, passing to the limit in \eqref{e5.28a} as $m\to\infty$ and taking into
account the relations \eqref{e5.29}-\eqref{e5.31}, we obtain the desired relation \eqref{e5.6}.

\noindent
{\sl Step 2.} Now we prove the relation \eqref{e5.6} without assuming the condition \eqref{e4.31}, that is, in the case where $h(\lambda)\equiv 1$ and $t(\lambda)$ is an arbitrary nonnegative trigonometric polynomial.
To this end, we denote $c:=\max_{\la\in\Lambda}t(\la)>0$, and consider the
trigonometric polynomial $\hat t(\la):=(1/c)t(\la)$.
Observe that the polynomial $\hat t(\la)$ satisfies the condition
\eqref{e4.31}, that is, $\hat t(\la)\leq 1$
for all $\la\in\Lambda$. Therefore for $\hat t(\la)$ the relation
\eqref{e5.6} is satisfied, that is, we have
\beq
\label{e5.32a}
\lim_{n\to\infty}\frac{\sigma_n^2(f \hat t^\al)}{\sigma_n^2(f)}= G(\hat t^\al).
\eeq
On the other hand, by Propositions \ref{pp3}(c) and \ref{p4.2}(a), we have
\beq
\label{e5.33}
\sigma_n^2(f {\hat t}^\al)= c^{-\al}\sigma_n^2(f t^\al)\q {\rm and} \q
G(f \hat t^\al)= c^{-\al}G(f t^\al).
\eeq

From \eqref{e5.32a} and \eqref{e5.33} we get
\beq
\nonumber
\frac1{c^\al}\lim_{n\to\infty}\frac{\sigma_n^2(f t^\al)}{\sigma_n^2(f)}
=c^{-\al} G(t^\al),
\eeq
which is equivalent to \eqref{e5.6}.

\noindent
{\sl Step 3.} Now we prove the relation \eqref{e5.6}  in the case where $h(\lambda)\equiv 1$ and $t(\lambda)$ is an arbitrary trigonometric polynomial.
Denoting $\check{t}(\la):=t^2(\la)$, we can write
\beq
\label{e5.45}
|t(\la)|^\al=\left(t^2(\la)\right)^{^\al/2}=\left(\check{t}(\la)\right)^{^\al/2},
\q \la\in\Lambda.
\eeq
Since $\check{t}(\la)$ is a nonnegative trigonometric polynomial,
according to Step 2, for $\check{t}(\la)$ the relation \eqref{e5.6}
is satisfied, that is, we have

\bea
\nonumber
\lim_{n\to\infty}\frac{\sigma_n^2(f\check{t}^{\al/2})}{\sigma_n^2(f)}= G(\check{t}^{\al/2}),
\eea
and, in view of \eqref{e5.45}, we get
\bea
\label{e5.46}
\lim_{n\to\infty}\frac{\sigma_n^2(f|t|^{\al})}{\sigma_n^2(f)}= G(|t|^{\al}),
\eea
and the result follows.

\noindent
{\sl Step 4.} Finally, we prove the relation \eqref{e5.6}, and thus the theorem,
in the general case, that is, when $h(\la)$ is an arbitrary function from the class $B_+^-$ and $t$ is an arbitrary trigonometric polynomial.
By Step 3, we have $g_1:=|t|^\al\in\mathcal{M}_f$ and $fg_1=f|t|^\al\in\mathcal{F}$. Also, according to Theorem B with $fg_1$
as $f$, the function $g_2:=h$ belongs to the class $\mathcal{M}_{fg_1}$.
Thus, the functions $g_1$ and $g_2$ satisfy the conditions of Proposition
\ref{ppg}, and hence
$g:=g_1g_2=h|t|^{\al}\in\mathcal{M}_f$ and $fg=fh|t|^{\al}\in\mathcal{F}$.
This completes the proof of Theorem \ref{T1}.
\end{proof}

\begin{proof}[Proof of Corollary \ref{ec5.2}]
For $m=1$ the corollary coincides with Theorem \ref{T1}.
The result then follows from inductive arguments and Proposition
\ref{ppg}.

\end{proof}
\begin{proof}[Proof of Theorem \ref{T2}]
Observe first that by Proposition  \ref{p4.2}(c) we have $G(t^{-\al})>0$.
Next, since $k+1-\al>0$ (recall that $k$ is the integer part of $\al$),
from the equality
$ft^{-\al}=ft^{-(k+1)}t^{(k+1-\al)}$, we have $ft^{-\al}\in L^1(\Lambda)$.
Denote $g_1:=t^{-(k+1)}$ and $g_2:=t^{(k+1-\al)}$. It follows from
Lemma \ref{l5-0} that the function $g_1$ satisfies the conditions:
$g_1\in\mathcal{M}_f$ and $fg_1\in\mathcal{F}$. Also, according to Theorem
\ref{T1}, we have  $g_2\in\mathcal{M}_{fg_1}$.
Therefore, by Proposition \ref{ppg}, we have
$g_1g_2=t^{-\al}\in\mathcal{M}_f$ and $fg_1g_2=ft^{-\al}\in\mathcal{F}$.
Applying Proposition \ref{ppg} now to the functions $t^{-\al}$ and $h$,
we conclude that
$g:=ht^{-\al}\in\mathcal{M}_f$ and $fg\in\mathcal{F}$.
Theorem \ref{T2} is proved.
\end{proof}


\begin{proof}[Proof of Theorem \ref{T2c}]
We first prove assertion a), that is, that the function $h:=\hat f/f$
belongs to the class $B_+^-$. To do this we denote $c:=\min_{1\leq j\leq k}c_j$,
where $c_j$ is as in \eqref{s3.8}. Then, in view of \eqref{s3.8},
for $\vs=c/2$ and each $j=1,2,\ldots, k$, there is a number $\de_j=\de_j(\vs)>0$
such that $|h(\la)-c_j|<\vs$ for all $\la\in O_{\de_j}(\la_j)$, or equivalently
\beq
\label{n4.49}
c_j-c/2<h(\la)<c_j+c/2 \q \text{for all} \q \la\in O_{\de_j}(\la_j), \q j=1,2,\ldots, k.
\eeq
Denote $C:=\max_{1\leq j\leq k}c_j$, $\la_0=-\pi$, $\De\la_j=\la_j-\la_{j-1}$,
and $\de:=\min_{1\leq j\leq k}\{\de_j, \De\la_j/3\}$.
Then, in view of the inequalities $c\leq c_j\leq C$ and $\de_j\leq \de$, from
\eqref{n4.49} we obtain
\beq
\label{n4.50}
c/2<h(\la)<3C/2 \q \text{for all} \q \la\in\cup_{j=1}^kO_\de(\la_j).
\eeq
Thus, the function $h(\la)$ is bounded away from zero and infinity in the set
$\cup_{j=1}^kO_\de(\la_j)$.
On the other hand, according to the conditions 1) and 4) of the theorem,
positive constants $m$ and $M$ can be found to satisfy
\beq
\label{n4.51}
f(\la)\leq M, \,\, \hat f(\la)\leq M \q \text{for all} \q \la\in \Lambda.
\eeq
and
\beq
\label{n4.52}
f(\la)\geq m, \,\, \hat f(\la)\geq m \q \text{for all} \q
\la\in \Lambda\setminus \cup_{j=1}^kO_\de(\la_j).
\eeq
Therefore
\beq
\label{n4.53}
m/M\leq h(\la)\leq M/m \q \text{for all} \q \la\in \Lambda\setminus \cup_{j=1}^kO_\de(\la_j).
\eeq
Hence, denoting $\hat m:=\min\{c/2, m/M\}$ and $\hat M:=\max\{3C/2, M/m\}$,
in view of \eqref{n4.50} and \eqref{n4.53} we obtain
\beq
\nonumber
\hat m<h(\la)<\hat M \q \text{for all} \q \la\in \Lambda.
\eeq
Thus, the function  $h(\la)$ ia bounded away from zero and infinity.
Therefore, $h(\la)$, being a ratio of two functions from the class $B^-$,
by Proposition \ref{p3.3} c), also belongs to the class $B^-$. Therefore,
$h(\la)\in B_+^-$, and hence $G(h)>0$.

To prove assertion b), assume that the process with spectral density $f(\la)$
is nondeterministic, that is, $G(f)>0$. Then taking into account that  $\hat f(\la)= f(\la)h(\la)$,
we have $G(\hat f)=G(f)G(h)>0$, implying that the process with spectral density
$\hat f(\la)$ is also nondeterministic. So, in view of \eqref{c013}, we have
\bea
\label{n4.54}
\lim_{n\to\f}\si_n^2(\hat f)=2\pi G(\hat f)\q \text{and} \q
\lim_{n\to\f}\si_n^2(f)=2\pi G(f).
\eea
From \eqref{n4.54} and Proposition \ref{p4.2}(a) we easily obtain the relation \eqref{g3.6e}.
Besides, from the first relation in \eqref{n4.54} and Proposition \ref{p4.1}(b),
we infer that the sequence $\{\si_n(\hat f)\}$ is weakly varying.
This completes the proof of assertion b) of the theorem.

The assertion c) of the theorem immediately follows from Theorem \ref{T1}.
Indeed, if $f\in\mathcal{F}$, then applying Theorem \ref{T1} with $t(\la)\equiv1$ (that is, $g(\la)=h(\la)$), we conclude that $\hat f= fh$
is the spectral density of a deterministic process with a weakly varying
sequence $\{\si_n(\hat f)\}$ of prediction errors and the relation \eqref{g3.6e} holds.
This completes the proof of Theorem \ref{T2c}.
\end{proof}

\begin{proof}[Proof of Theorem \ref{T1b}]
Let the polynomial $q(\la)$ in \eqref{e5.5b} be of degree $m\geq 1$, and let
$\la_i\in(-\pi,\pi]$
be the zeros of $q(\la)$ of multiplicities $m_i\in\mathbb{N}$
$(i=1,2,\ldots, k, k\in\mathbb{N})$, respectively. Then by the Fundamental Theorem of Algebra we
can write
$$q(\la)=c_0(\la-\la_1)^{m_1}(\la-\la_2)^{m_2}\cdots(\la-\la_k)^{m_k}\hat q(\la),\q \la\in\Lambda,$$
where $\hat q(\la)$ consists of a product of linear binomial factors with zeros outside $(-\pi,\pi]$
and quadratic trinomial factors, which are positive on $\mathbb{R}$, implying that
$\hat q(\la)\in B_+^-$.

Consider the trigonometric polynomial:
$$t(\la)=\sin^{m_1}(\la-\la_1)\sin^{m_2}(\la-\la_2)\cdots\sin^{m_k}(\la-\la_k)
,\q \la\in\Lambda,$$
and the function
\beq
\label{n4.55}
\hat h(\la):=\frac{|q(\la)|^\al}{|t(\la)|^\al},\q \la\in\Lambda.
\eeq
From the relation $\sin(\la-\la_i)\sim(\la-\la_i)$ as $\la\to\la_i$ it follows that with some positive constants $c_i$, 
$|q(\la)|^\al\sim c_i|t(\la)|^\al$ as $\la\to\la_i$ $(i=1,2,\ldots, k)$.
Hence the functions $|q(\la)|^\al$ and $|t(\la)|^\al$, with regard to their continuity,
satisfy the conditions of Theorem \ref{T2c}. Therefore,
$\hat h(\la)\in B_+^-$ and, according to Proposition \ref{p3.3} d),
we have $h(\la) \hat h(\la)\in B_+^-$.

Now we can apply Theorem \ref{T2} to conclude the function $g(\la)$ in \eqref{e5.5b} 
is the spectral density of a nondeterministic process and to obtain
\beaa
\nonumber
\lim_{n\to\infty}\frac{\sigma_n^2(fg)}{\sigma_n^2(f)}
&=&\lim_{n\to\infty}\frac{\sigma_n^2(fh|q|^\al)}{\sigma_n^2(f)}
=\lim_{n\to\infty}\frac{\sigma_n^2(fh\hat h|t|^\al)}{\sigma_n^2(f)}\\
&=&G(h\hat h|t|^\al)=G(h|q|^\al)=G(g)>0.
\eeaa
Here the first and the fifth relations follow from \eqref{e5.5b},
the second and the fourth relations follow from \eqref{n4.55},
and the third relation follows from Theorem \ref{T2}.

This completes the proof of Theorem \ref{T1b}.
\end{proof}

\s{Examples}
\label{Examples}

In this section we discuss examples demonstrating the obtained results.

\begin{exa}
\label{ex1}
{\rm Let the function $g(\la)$ ($\la\in\Lambda$) be as in \eqref{e5.5a} with $h(\la)=1$
and $t(\la)=\sin(\la-\la_0)$, where $\la_0$ is an arbitrary point from
$[-\pi,\pi]$, that is, $g(\la)=|\sin(\la-\la_0)|^{\al}$, $\al\in\mathbb{R}$.
Then, according to Example 4.4 of Babayan et al. \cite{BGT},
for the geometric mean of $\sin^{2}(\la-\la_0)$ we have
\beq
\label{k742g}
G(\sin^{2}(\la-\la_0))=\frac14.
\eeq
According to Proposition \ref{p4.2}(a) and \eqref{k742g},
for the geometric mean of $g(\la)$, we obtain
\beq
\label{g4}
G(g)=G(|\sin(\la-\la_0)|^{\al})=G\left(\left(\sin^{2}(\la-\la_0)\right)^{\al/2}\right)=
G^{\al/2}(\sin^2(\la-\la_0))=\frac1{2^\al},
\eeq
and in view of \eqref{e5.6}, we get
\beq
\label{k76}
\nonumber
\lim_{n\to\f}\frac{\si^2_{n}(fg)}{\si^2_{n}(f)} =G(g)=\frac1{2^\al}.
\eeq
Thus, multiplying the spectral density $f(\la)$ by the function
$g(\la)=|\sin(\la-\la_0)|^{\al}$
yields a $2^\al$-fold asymptotic reduction of the prediction error.}
\end{exa}

\begin{exa}
\label{ex2}
{\rm Let the function $g(\la)$ be as in \eqref{e5.5b} with $h(\la)=1$
and $q(\la)=\la$, that is, $g(\la)=|\la|^\al$, $\al\in\mathbb{R}$.
By direct calculation we obtain
\beaa
\ln G(g)=\frac1{2\pi}\inl\ln |\la|^\al\,d\la =\frac\al\pi\int_0^\pi\ln\la\,d\la
=\al\ln(\pi/e).
\eeaa
Therefore
\beaa
G(g)= 
\left(\pi/ e\right)^\al \approx (1.156)^\al,
\eeaa
and in view of \eqref{e5.6}, we get
\beq
\label{k72a}
\nonumber
\lim_{n\to\f}\frac{\si^2_{n}(fg)}{\si^2_{n}(f)} =G(g)
=\left(\frac\pi e\right)^\al \approx (1.156)^\al.
\eeq
Thus, multiplying the spectral density $f(\la)$ by the function
$g(\la)=|\la|^\al$ multiplies the prediction error asymptotically
by $(\pi/e)^\al \approx (1.156)^\al$.

It follows from Proposition \ref{pp3}(d) that the same asymptotic
is true for the prediction error with spectral density
$\bar g(\la)=|\la-\la_0|^\al$, $\la_0\in [-\pi,\pi]$.}
\end{exa}

\begin{exa}
\label{ex3}
{\rm
We first analyze 
the Pollaczek-Szeg\H{o} function
$f_a(\la)$ given by \eqref{nd4} (cf.  Pollaczek \cite{Po} and Szeg\H{o} \cite{S1}). We have
\beq
\label{nd41}
f_a(\la)=\frac{2e^{2\la\varphi(\la)}e^{-\pi\varphi(\la)}}{e^{\pi\varphi(\la)}+e^{-\pi\varphi(\la)}}
=\frac{2e^{2\la\varphi(\la)}}{e^{2\pi\varphi(\la)}+1},
\q  0\leq\la\leq\pi, \q \varphi(\la):=\varphi_a(\la)=(a/2)\cot\la. 
\eeq
Observe that $\varphi(\la)\to+\f$ as $\la\to 0^+$, and we have 
\beq
\label{nd42}
\varphi(\la)\sim a/(2\la), \q e^{2\la\varphi(\la)}\sim e^a, \q
e^{2\pi\varphi(\la)}+1 \sim e^{a\pi/\la} \q {\rm as} \q \la\to0^+.
\eeq
Taking into account that $f_a(\la)$ is an even function, from \eqref{nd41}
and \eqref{nd42} we obtain the following asymptotic relation for $f_a(\la)$
in a vicinity of the point $\la=0$.
\beq
\label{nd43}
f_a(\la)\sim 2e^a\exp\left\{-{a\pi}/{|\la|}\right\}\q {\rm as}
\q \la\to0.
\eeq
Next, observe that $\varphi(\la)\to-\f$ as $\la\to \pi$, and we have 
\beq
\label{nd44}
\varphi(\la)=-\varphi(\pi-\la) \sim (-a/2)(\pi-\la), \q 2\la\varphi(\la)\sim -a\pi/(\pi-\la), \q
{\rm as} \q \la\to\pi.
\eeq
In view of \eqref{nd41} and \eqref{nd44} we obtain the following asymptotic
relation for the function $f_a(\la)$ in a vicinity of the point $\la=\pi$.
\beq
\label{nd45}
f_a(\la)\sim 2e^{2\la\varphi(\la)} \sim 2\exp\left\{-{a\pi}/{(\pi-\la)}\right\}\q {\rm as}
\q \la\to\pi.
\eeq
Putting together \eqref{nd43} and \eqref{nd45}, and taking into account
evenness of $f_a(\la)$, we conclude that
\beq \label{t2}
f_a(\la)\sim
\left \{
\begin{array}{ll}
 2e^a\exp\left\{-{a\pi}/{|\la|}\right\} & \mbox{as $\la\to0$},\\
2\exp\left\{-{a\pi}/{(\pi-|\la|)}\right\} & \mbox{as $\la\to\pm\pi$},
\end{array}
\right.
\eeq
Thus, the function $f_a(\la)$ is positive everywhere except for points
$\la=0, \pm\pi,$ and has a very high order of contact with zero at these points, so that Szeg\H{o}'s condition \eqref{S} is violated implying that $G(f_a)=0$. Also, observe that $f_a(\la)$ is infinitely differentiable at
all points of the segment $[-\pi,\pi]$ including the points $\la=0, \pm\pi,$
and attains it maximum value of 1 at the points $\pm\pi/2$.
For some specific values of the parameter $a$ the graph of the function $f_a(\la)$ is represented in Figure 1a).


\begin{figure}[ht]%
\centering
\includegraphics[width=0.8\textwidth]{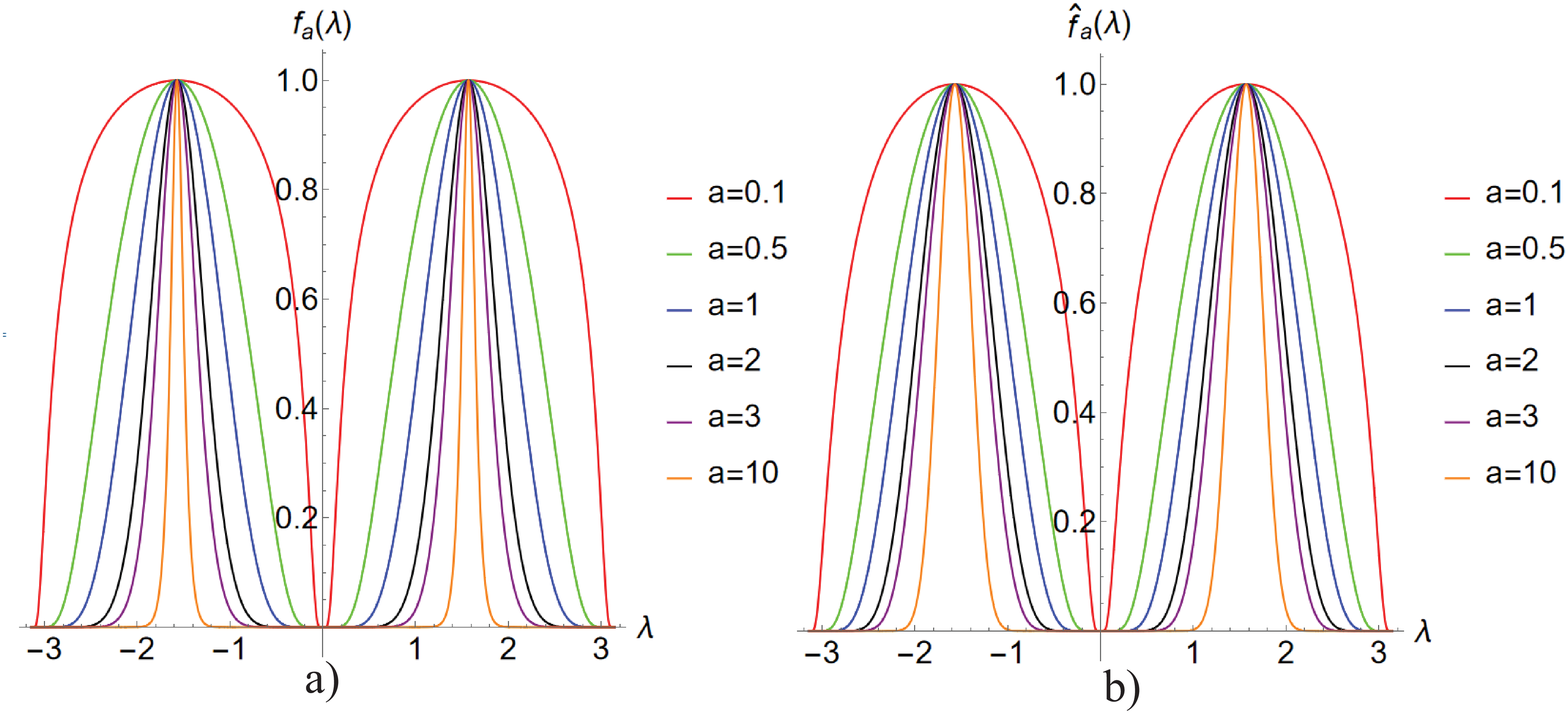}
\caption{a) Graph of the function $f_a(\lambda)$.
b) Graph of the function $\hat f_a(\lambda)$.}
\label{fig1}
\end{figure}



For $a>0$ and  $\la\in [-\pi,\pi]$, consider the pair of functions
$\hat f_1(\la)$ and $\hat f_2(\la)$ defined by formulas:
\beq
\label{ff}
\hat f_1(\la):=\exp\left\{-{a\pi}/{|\la|}\right\}, \q
\hat f_2(\la):=\exp\left\{-{a\pi}/{(\pi-|\la|)}\right\}.
\eeq

Observe that the function $\hat f_1(\la)$ is positive everywhere except
for point $\la=0$ at which it has the same order of contact with zero
as $f_a(\la)$, and hence $G(\hat f_1)=0$. Also, $\hat f_1(\la)$ is infinitely differentiable at
all points of the segment $[-\pi,\pi]$  except for the points
$\la=\pm\pi,$ where it attains its maximum value equal to $e^{-a}$.
As for the function $\hat f_2(\la)$, it is positive everywhere except
for points $\la=\pm\pi,$ at which it has the same order of contact with
zero as $f_a(\la)$,
and hence $G(\hat f_2)=0$. Also, $\hat f_2(\la)$ is infinitely differentiable
at all points of the segment $[-\pi,\pi]$  except for the point $\la=0$,
where it attains its maximum value equal to $e^{-a}$.
For some specific values of the parameter $a$ the graphs of functions
$\hat f_1(\la)$ and $\hat f_2(\la)$ are represented in Figure 2.

\begin{figure}[ht]%
\centering
\includegraphics[width=0.7\textwidth]{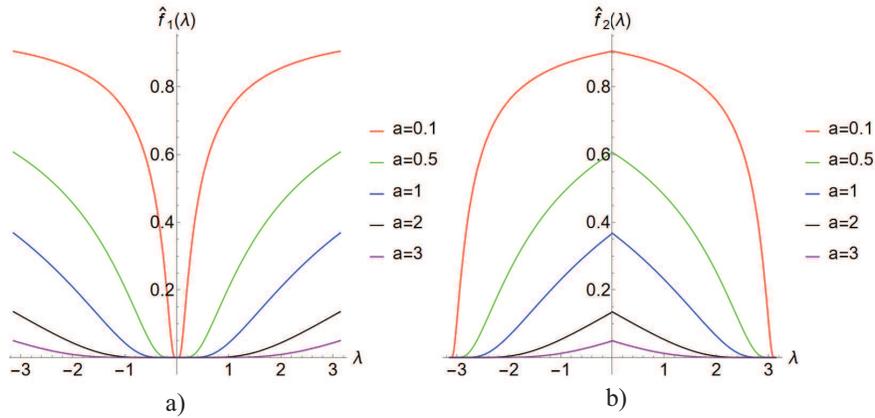}
\caption{a) Graph of the function $\hat f_1(\lambda)$.
b) Graph of the function $\hat f_2(\lambda)$.} 
\label{fig2}
\end{figure}



Denote by $\hat f_a(\la)$ the product of functions $\hat f_1(\la)$
and $\hat f_2(\la)$ defined in \eqref{ff} and normalized by the factor $e^{4a}$:
\beq
\label{snd5}
\hat f_a(\la):=e^{4a}\hat f_1(\la)\hat f_2(\la)
=e^{4a}\exp\left\{-{a\pi^2}/{(|\la|(\pi-|\la|))}\right\},
\eeq
and observe that $\hat f_a(\la)$ behaves similar to $f_a(\la)$.
Indeed, the function $\hat f_a(\la)$ also is positive everywhere except for points $\la=0, \pm\pi,$ it is infinitely differentiable at all points of the segment $[-\pi,\pi]$ including the points $\la=0, \pm\pi,$ and attains it maximum value of 1 at the points $\pm\pi/2$.
Also, in view of \eqref{t2} and \eqref{snd5}, at points $\la=0, \pm\pi$ the function $\hat f_a(\la)$ has the same order of zeros as $f_a(\la)$, and hence $G(\hat f_a)=0$.
Thus, the process $X(t)$ with spectral density $\hat f_a(\la)$ is deterministic.
For some specific values of the parameter $a$ the graph of the function $\hat f_a(\la)$ is represented in Figure 1b).


The functions $f_a(\la)$ and $\hat f_a(\la)$ defined
by \eqref{nd4} and \eqref{snd5}, respectively, satisfy the conditions of
Theorem \ref{T2c}. 
Therefore, we have (see \eqref{g3.6e})
\beq
\label{snd6}
\lim_{n\to\f}\frac{\si_n^2(\hat f_a)}{\si_n^2(f_a)}=G(\hat f_a/f_a):
=\hat C(a)>0.
\eeq
In view of \eqref{nd6} and \eqref{snd6} we have
\beq
\label{snd7}
\si^2_n(\hat f_a)\sim C(a) \cd n^{-a}
\q {\rm as} \q n\to\f.
\eeq
where
\beq
\label{snd8}
C(a):=\frac{\Gamma^2\left(({a+1)}/2\right)}{\pi 2^{2-a}}\cd \hat C(a). 
\eeq
}
\end{exa}

The values of the constants $\hat C(a)$ and $C(a)$ for some specific values of the parameter $a$ are given in Table 1.

\vskip 3mm
\begin{center} {\bf Table 1.} The values of constants $\hat C(a)$ and $C(a)$\\
\begin{tabular}{llll} \toprule
    {$a$} & {$\frac{\Gamma^2\left((a+1)/2\right)}{\pi 2^{2-a}}$} & {$\hat C(a)$} & {$C(a)$} \\ \midrule
    0.1  & 0.223  & 0.797   & 0.178  \\
    0.5  & 0.169  & 1.113   & 0.188  \\
    1.0  & 0.159  & 2.545   & 0.406  \\
    1.5  & 0.185  & 6.446   & 1.193  \\ 
    2.0  & 0.250  & 16.830  & 4.214   \\
    3.0  & 0.637  & 119.220 & 76.379  \\
    3.3  & 0.902   & 215.715 & 194.656  \\ \midrule
    3.4  & 1.020   & 263.173 & 268.375   \\
    5.0  & 10.186  & 6128.990& 62429.000  \\
    10.0   & 223256 & 1.104 $\cdot 10^8$ & 2.428 $\cdot 10^{13}$  \\
    \bottomrule
\end{tabular}
\end{center}




Now we compare the prediction errors $\si^2_n(\hat f_1)$ and
$\si^2_n(\hat f_2)$ with $\si^2_n(f_a)$.
To this end, observe first that
the function
$g_1(\la):=f_a(\la)/\hat f_1(\la)$ has a very high order of contact
with zero at points $\la=\pm\pi,$ so that Szeg\H{o}'s condition \eqref{S} is violated implying that $G(g_1)=0$. Besides, the function $g_1(\la)$ is continuous
on $[-\pi,\pi]$, and hence $g_1\in B^-$. Therefore, according to Corollary
4.5 of Babayan et al. \cite{BGT}, we have
\beq
\label{ff1}
\si^2_n( f_a)= o\left(\si^2_n(\hat f_1)\right)
\q {\rm as} \q n\to\f.
\eeq
Similar arguments applied to the function $f_2(\la)$ yield
\beq
\label{ff2}
\si^2_n( f_a)= o\left(\si^2_n(\hat f_2)\right)
\q {\rm as} \q n\to\f.
\eeq
The relations  \eqref{ff1} and \eqref{ff2} show that the rate of
convergence to zero of the prediction errors $\si^2_n(\hat f_1)$
and $\si^2_n(\hat f_2)$ is less than the one for $\si^2_n(f_a)$, that is, the power rate of convergence $n^{-a}$ (see \eqref{snd7}).
Thus, the rate of convergence $n^{-a}$ is due to the joint contribution
of all zeros $\la=0,\pm\pi$ of the function $f_a(\la)$, 
whereas
each of these zeros separately does not guarantee the rate of convergence $n^{-a}$.

\end{document}